\newif{\ifarxiv}
\arxivtrue
\ifarxiv
\documentclass{article}
\pdfoutput=1 % for arXiv to recognize this as pdfLaTeX
\usepackage[T1]{fontenc}
\usepackage{amsmath,amsfonts,amssymb,amsthm,enumitem}
\usepackage{authblk}
\else
\documentclass{amsart}
\usepackage[T1]{fontenc}   %%%%%%%Please do not change
\usepackage{amsfonts,amssymb,enumitem}

\fi

\newtheorem{theorem}{Theorem}[section]
\newtheorem{lemma}[theorem]{Lemma}

\newtheorem{corollary}[theorem]{Corollary}

\ifarxiv
\relax
\else
\theoremstyle{definition}
\fi
\newtheorem{definition}[theorem]{Definition}

\ifarxiv
\relax
\else
\numberwithin{equation}{section}
\fi

\newcommand\Open{{\mathcal O}}
\newcommand{\interior}[1]{int ({#1})} %{{\buildrel \circ \over {#1}}}
\newcommand\real{\mathbb{R}}
\newcommand\Rp{\real_+}
\newcommand\creal{\overline\real_+}

\newcommand\nat{\mathbb{N}}
\newcommand\pow{\mathbb{P}}
\newcommand\Pfin{\pow_{\text{fin}}}
\newcommand\dc{\mathop{\downarrow}}

\newcommand\upc{\mathop{\uparrow}}

\newcommand\uuarrow{\rlap{$\uparrow$}\raise.5ex\hbox{$\uparrow$}}%%
\newcommand\ddarrow{\rlap{$\downarrow$}\raise.5ex\hbox{$\downarrow$}}%%

\newcommand\eqdef{\mathrel{\buildrel \text{def}\over=}}
\newcommand\diff{\smallsetminus}

\newcommand\Val{\mathbf V}
\newcommand\wk{\mathrm{w}}
\newcommand\pw{\mathrm{p}}
\newcommand\fin{\mathrm{f}}
\newcommand\ugame[1]{\mathfrak{u}_{#1}}

\begin{document}

\ifarxiv
\relax
\else
%%%%%%%%%%%%%%%%%%%%%%%%%%%%%%%%%%%%%%%%%%%%%%%%%%%%%%%%%%%%
%%%%%%%%%%%%%%%%%%%%%%%%%%%%%%%%%%%%%%%%%%%%%%%%%%%%%%%%%%%%
% This a placeholder for the TOPLOGY PROCEEDINGS logo %%%%%%
\noindent                                             %%%%%%
\begin{picture}(150,36)                               %%%%%%
\put(5,20){\tiny{Submitted to}}                       %%%%%%
\put(5,7){\textbf{Topology Proceedings}}              %%%%%%
\put(0,0){\framebox(140,34){}}                        %%%%%%
\put(2,2){\framebox(136,30){}}                        %%%%%%
\end{picture}                                        %%%%%%
%%%%%%%%%%%%%%%%%%%%%%%%%%%%%%%%%%%%%%%%%%%%%%%%%%%%%%%%%%%%
%%%%%%%%%%%%%%%%%%%%%%%%%%%%%%%%%%%%%%%%%%%%%%%%%%%%%%%%%%%%
\vspace{0.5in}
\fi

\renewcommand{\bf}{\bfseries}
\renewcommand{\sc}{\scshape}
%insert defs/styles
\vspace{0.5in}

\title{Probabilistic Powerdomains and Quasi-Continuous Domains}

\newcommand\me{Jean Goubault-Larrecq}
\newcommand\where{Universit\'e Paris-Saclay, ENS Paris-Saclay, CNRS, 
  LSV, 91190, Gif-sur-Yvette, France.}
\newcommand\eml{goubault@lsv.fr}
\ifarxiv
\author[$\dagger$]{\me}
\affil[$\dagger$]{\where \quad \texttt{\eml}}
\else
%    Information for first author:
\author{\me}
\address{\where}
\email{\eml}
\fi

\ifarxiv
\relax
\else
%    General info
%%%%%%%%%%%%%%%%%%%%%%%%%%%%%%%%%%%%%%%%%%%%%%%%%%%
\subjclass[2020]{Primary
54G99; % General topology - Peculiar spaces - None of the above, but in
      % this section
secondary 
28E99. %		Miscellaneous topics in measure theory
% None of the above, but in this section
% 68R99, % Computer science - Discrete mathematics in relation to
%       % computer science - None of the above, but in this section
% 06A07, % Order, lattices, ordered algebraic structures - Ordered sets -
      % Combinatorics of partially ordered sets
% 06B30% Order, lattices, ordered algebraic structures - Lattices -
       % Topological lattices, order topologies
%  54D99 % General topology - Fairly general properties - None of the
%       % above, but in this section
% 54E99 % General topology - Spaces with richer structure - None of the
%       % above, but in this section
% 06F30 % Order, lattices, ordered algebraic structures - Lattices -
%       % Ordered structures - Topological lattices, order topologies
}
%                                                                                                                           %
%         Please use the current 2010 Mathematics Subject Classification:             %
%         http://www.ams.org/mathscinet/msc/                                                        %
%         http://www.zentralblatt-math.org/msc/en/                                                 %
%%%%%%%%%%%%%%%%%%%%%%%%%%%%%%%%%%%%%%%%%%%%%%%%%%%

\keywords{Probabilistic powerdomain, quasi-continuous domain, locally
  finitary compact space, Scott topology, weak topology}
% \thanks{This research was supported by the BRAVAS project, number
% ANR-17-xxxx, of the French National Research Agency (ANR)}
% \thanks{The author was supported by grant ANR-17-CE40-0028 of the
%   French National Research Agency ANR (project BRAVAS)} % JGL, 21 dec 2017
\fi

\newcommand\abs{%
  The probabilistic powerdomain $\Val X$ on a space $X$ is the space
  of all continuous valuations on $X$.  We show that, for every
  quasi-continuous domain $X$, $\Val X$ is again a quasi-continuous
  domain, and that the Scott and weak topologies then agree on
  $\Val X$.  This also applies to the subspaces of probability and
  subprobability valuations on $X$.  We also show that the Scott and
  weak topologies on the $\Val X$ may differ when $X$ is not
  quasi-continuous, and we give a simple, compact Hausdorff
  counterexample.
}

\maketitle

\begin{abstract}
  \abs
\end{abstract}

\section{\bf Introduction}
\label{sec:intro}

Continuous valuations are an alternative to measures, which are
popular in computer science, and notably in the semantics of
programming languages \cite{jones89,Jones:proba}.  The space of all
continuous valuations on a topological space $X$ is called the
probabilistic powerdomain $\Val X$ on $X$.  It is known that the
probabilistic powerdomain of a directed-complete partial order (dcpo)
is a dcpo again, in short, $\Val$ preserves dcpos; similarly, $\Val$
preserves continuous dcpos, but fails to preserve complete lattices
and bc-domains.  All that was proved by Jones
\cite{jones89,Jones:proba}.  It is unknown whether $\Val$ preserves
RB-domains or FS-domains, except in special cases
\cite{JT:troublesome}.  On the positive side, $\Val$ preserves stably
compact spaces \cite{JT:troublesome,AMJK:scs:prob}, QRB-domains
\cite{JGL:qrb,GLJ:QRB=QFS}, and coherent quasi-continuous dcpos
\cite{LK:prob}.  (The latter two results are equivalent, since
QRB-domains coincide with coherent quasi-continuous dcpos
\cite{LX:QRB=QFS,GLJ:QRB=QFS}, and also with Li and Xu's QFS-domains
\cite{LiXu:QFS}.)

Lyu and Kou \cite{LK:prob} asked whether coherence was required, in
other words, whether $\Val$ preserves quasi-continuous, not
necessarily coherent, dcpos.  We show that this is indeed the case,
and that, in this case, the Scott and weak topologies agree on the
probabilistic powerdomain.  We show this in
Section~\ref{sec:bf-main-theorem}, after a few preliminaries: general
preliminaries in Section~\ref{sec:bf-preliminaries}, some required
material due to Heckmann on so-called point-continuous valuations in
Section~\ref{sec:simple-valuations}, and a useful lemma on capacities
in Section~\ref{sec:capacities}.  We refine the result and we handle
the case of probability continuous valuations in
Section~\ref{sec:case-prob-cont}.

Since every continuous dcpo is quasi-continuous, the coincidence of
the Scott and weak topologies on $\Val X$, where $X$ is
quasi-continuous, generalizes a result of Kirch
\cite[Satz~8.6]{Kirch:bewertung}, see also
\cite[Satz~4.10]{Tix:bewertung}, according to which the Scott and weak
topologies on $\Val X$ agree for every continuous dcpo $X$.
Alvarez-Manilla, Jung and Keimel asked whether they agree on
$\Val_{\leq 1} X$ for every stably compact space $X$ \cite[Section~5,
second open problem]{AMJK:scs:prob}.  We will show that this is not
the case, through a simple, compact Hausdorff example in
Section~\ref{sec:scott-weak-topol}.  Hence the situation with
quasi-continuous domains is probably rather exceptional.

\section{\bf Preliminaries}
\label{sec:bf-preliminaries}

We refer to \cite{GHKLMS:contlatt,JGL-topology} on domain theory and
point-set topology, specially non-Hausdorff topology.  Compactness
does not involve separation.

A \emph{dcpo} (directed-complete partial order) is a poset $P$ in
which every directed family $D$ has a supremum $\sup D$.  A
\emph{Scott-open subset} of $P$ is a subset $U$ that is
\emph{upwards-closed} (for every $x \in U$ and every $y$ such that
$x \leq y$, $y$ is in $U$) and is such that, for every directed family
$D$ in $P$, if $\sup D \in U$ then $D$ intersects $U$.  The
Scott-open subsets of $P$ form a topology called the \emph{Scott
  topology}.

Every complete lattice is a dcpo.  For example,
$\creal \eqdef \Rp \cup \{\infty\}$, with the usual ordering that
places $\infty$ above all non-negative real numbers, is a dcpo.  The
family of open subsets $\Open X$ of a topological space is a dcpo,
too.

A \emph{Scott-continuous map} $f \colon P \to Q$ between dcpos is a
monotonic map that preserves suprema of directed families.  A map from
$P$ to $Q$ is Scott-continuous if and only if it is continuous with
respect to the Scott topologies on $P$ and $Q$.

A \emph{valuation} $\nu$ on a topological space $X$ is a strict,
modular, monotonic map from $\Open X$ to $\creal$.  That $\nu$ is
\emph{strict} means that $\nu (\emptyset)=0$.  That it is
\emph{modular} means that
$\nu (U) + \nu (V) = \nu (U \cup V) + \nu (U \cap V)$ for all open
subsets $U$ and $V$.  A \emph{continuous valuation} is a valuation
that is Scott-continuous from $\Open X$ to $\creal$.

Continuous valuations and measures are close cousins.  Every
$\tau$-smooth Borel measure defines a continuous valuation by
restricting it to $\Open X$; and every Borel measure on a hereditary
Lindel\"of space is $\tau$-smooth \cite{Adamski:measures}.  Conversely,
every continuous valuation on an LCS-complete space extends to a
measure on the Borel $\sigma$-algebra
\cite[Theorem~1.1]{dBGLJL:LCS}---an \emph{LCS-complete} space is any
subspace obtained as a $G_\delta$ subset of a locally compact sober
space.

% every locally finite continuous valuation on a locally compact sober
% space.  (A
% continuous valuation $\nu$ is \emph{locally finite} if and only every
% point $x \in X$ has an open neighborhood $U$ such that
% $\nu (U) < \infty$.)  This was proved by Alvarez-Manilla
% \cite[Theorem~3.27]{alvarez00}, see also
% \cite[Theorem~5.3]{KL:measureext}.

We write $\Val X$ for the dcpo of all continuous valuations on $X$,
ordered by the \emph{stochastic ordering}: $\mu \leq \nu$ if and only
if $\mu (U) \leq \nu (U)$ for every $U \in \Open X$.
$\Val_{\leq 1} X$ (resp., $\Val_1 X$) is the subdcpo of all
\emph{subprobability} (resp., \emph{probability}) continuous
valuations $\nu$, namely those such that $\nu (X) \leq 1$ (resp.,
$\nu (X)=1$).  We will usually write $\Val_\ast X$ to denote any of
those dcpos, where $\ast$ stands for nothing, ``$\leq 1$'', or
``$1$''.

The \emph{weak topology} on $\Val_\ast X$ is the coarsest one that
makes
$[r \ll U]_\ast \eqdef \{\nu \in \Val_\ast X \mid r \ll \nu (U)\}$
open for every $r \in \Rp$ and every $U \in \Open X$.  Here $\ll$ is
the so-called way-below relation on $\creal$; we have $r \ll s$ if and
only if $r=0$ or $r < s$.  The sets $[U > r]_\ast$ with
$U \in \Open X$ and $r \in \Rp \diff \{0\}$ form another subbase of
the weak topology, since $[U > r]_\ast = [r \ll U]_\ast$ if
$r \neq 0$, and $[0 \ll U]_\ast = \Val_\ast X$.  We write
$\Val_{\ast,\wk} X$ for $\Val_\ast X$ with the weak topology.  The
weak topology is coarser than the Scott topology of the stochastic
ordering $\leq$.

Every topological space $X$ has a \emph{specialization preordering}
$\leq$, defined by $x \leq y$ if and only if every open neighborhood
of $x$ contains $y$.  A $T_0$ space is one such that $\leq$ is an
ordering.  As examples, for every dcpo $P$, ordered by $\leq$, the
specialization preordering of $P$ with its Scott topology is $\leq$;
and the specialization preordering of $\Val_\ast X$ is the stochastic
ordering.

For every point $x \in X$, the closure of $\{x\}$ coincides with the
downward closure $\dc x \eqdef \{y \in X \mid y \leq x\}$ of $x$ in
the specialization preordering.  In general, we write $\dc A$ for the
downward closure of a set $A$, so that $\dc x = \dc \{x\}$.

A subset $A$ of a space $X$ is \emph{saturated} if and only if it is
equal to the intersection of its open neighborhoods, equivalently if
it is upwards-closed with respect to the specialization preordering
$\leq$.  We write $\upc A$ for the upward closure of $A$.

For every compact subset $K$ of $X$, $\upc K$ is compact saturated.
This is the case in particular if $K$ is finite: we call the sets of
the form $\upc E$, with $E$ finite, \emph{finitary compact}.  A space
$X$ is \emph{locally finitary compact} if and only if it has a base
consisting of interiors $\interior {\upc E}$ of finitary compact sets.

The standard definition of a quasi-continuous dcpo is through the
notion of a so-called way-below relation between finite subsets.  We
will instead use the following characterization
\cite[Exercise~8.3.39]{JGL-topology}: the quasi-continuous dcpos are
exactly the locally finitary compact, sober spaces.  Notably, every
locally finitary compact, sober space is a quasi-continuous dcpo in
its specialization ordering $\leq$; also, the topology is exactly the
Scott topology of $\leq$.

We have mentioned sober spaces a few times already.  A closed subset
$C$ of a space $X$ is \emph{irreducible} if and only if it is
non-empty and, for all closed subsets $C_1$ and $C_2$ of $X$, if
$C \subseteq C_1 \cup C_2$ then $C \subseteq C_1$ or
$C \subseteq C_2$.  The closures $\dc x$ of points are always
irreducible closed.  A \emph{sober space} is any $T_0$ space in which
the only irreducible closed subsets are closures of points.  $\creal$
is sober in its Scott topology.  Every quasi-continuous dcpo is sober
in its Scott topology (by our definition), every Hausdorff space is
sober; also, $\Val_\wk X$ is sober for every space $X$
\cite[Proposition~5.1]{heckmann96}.

The sober subspaces $Y$ of a sober space $X$ are exactly those that
are closed in the so-called \emph{Skula}, or \emph{strong} topology on
$X$ \cite[Corollary~3.5]{KL:dtopo}.  That topology is the coarsest one
that contains both the original open and the original closed sets as
open sets.  We note that $\Val_{\leq 1,\wk} X$ is closed in
$\Val_\wk X$, being the complement of $[X > 1]$.  Every closed set is
Skula-closed, so $\Val_{\leq 1, \wk} X$ is also a sober space.  Also,
$\Val_{1,\wk} X$ is the intersection of the closed set
$\Val_{\leq 1,\wk} X$ with the open sets $[X > 1-\epsilon]$,
$\epsilon > 0$, hence is also Skula-closed and therefore sober as
well.

The forgetful functor from the category of sober spaces and continuous
maps to the category of topological spaces has a left adjoint called
\emph{sobrification}.  Explicitly, this means that every topological
space $X$ has a sobrification $X^s$, which is a sober topological
space; there is a continuous map $\eta_X \colon X \to X^s$, called the
\emph{unit}; and every continuous map $f \colon X \to Y$ where $Y$ is
sober extends to a unique continuous map $\hat f \colon X^s \to Y$, in
the sense that $\hat f \circ \eta_X = f$.  Concretely, $X^s$ can be
realized as the space of all irreducible closed subsets of $X$, with a
suitable topology, and $\eta_X (x) \eqdef \dc x$.  By Proposition~3.4
of \cite{KL:dtopo}, given any subspace $Y$ of a sober space $X$, the
Skula-closure $cl_s (Y)$ of $Y$ in $X$ is also a sobrification of $Y$,
with $\eta_Y$ defined as the inclusion map.  In general, for a $T_0$
space $Y$, and a sober space $X$, together with a continuous map
$f \colon Y \to X$, $X$ is a sobrification of $Y$ with unit $f$ if and
only if $f$ is a topological embedding, with Skula-dense image
\cite[Proposition~3.2]{KL:dtopo}.

\section{\bf Simple and point-continuous valuations}
\label{sec:simple-valuations}

Among all the continuous valuations that exist on a space $X$, the
\emph{simple valuations} are those of the form $\sum_{x \in A} a_x
\delta_x$, where $A$ is a finite subset of $X$, $a_x \in \Rp$, and
$\delta_x$ is the Dirac mass, defined by $\delta_x (U) \eqdef 1$ if $x
\in U$, $0$ otherwise.  We let $\Val_{\ast,\fin} X$ be the subspace of
$\Val_{\ast,\wk} X$ that consists of its simple valuations.

Heckmann characterized the sobrification of $\Val_\fin X$ as being the
space $\Val_\pw X$ of so-called \emph{point-continuous valuations} on
$X$ \cite[Theorem~5.5]{heckmann96}, together with inclusion as unit.
Those are the valuations $\nu$ on $X$ that are continuous from
$\Open_\pw X$ to $\creal$.  $\Open_\pw X$ is the lattice of open
subsets of $X$ with the point topology, namely the coarsest topology
that makes $\{U \in \Open X \mid x \in U\}$ open for every point
$x \in X$.  We write $\Val_{\ast,\pw} X$ for the usual variants.
\begin{lemma}
  \label{lemma:VpX:dense}
  Let $X$ be a topological space.  $\Val_\fin X$ is Skula-dense in
  $\Val_\pw X$.
\end{lemma}
\begin{proof}
  By Proposition~3.2 of \cite{KL:dtopo}, cited earlier: since
  $\Val_\pw X$ is a sobrification of $\Val_\fin X$, with unit given by
  the inclusion map $i$, the image of $i$ must be Skula-dense.
\end{proof}

\begin{lemma}
  \label{lemma:mu:simple:approx:1}
  Let $X$ be a topological space and $\mathcal U$ be an open subset of
  $\Val_\pw X$.  For every $\nu \in \mathcal U$, there is a simple
  valuation $\nu'$ in $\mathcal U$ such that $\nu' \leq \nu$.
\end{lemma}
\begin{proof}
  $\mathcal U \cap \dc \nu$ is open in the Skula topology of
  $\Val_\wk X$, and is non-empty, since it contains $\nu$.  Using
  Lemma~\ref{lemma:VpX:dense}, it must contain an element $\nu'$ of
  $\Val_\fin X$.
\end{proof}

Heckmann also showed that, when $X$ is locally finitary compact,
\emph{all} continuous valuations are point-continuous, hence
$\Val_\wk X = \Val_\pw X$ \cite[Theorem~4.1]{heckmann96}.  Using that
information, we obtain the following.
\begin{lemma}
  \label{lemma:mu:simple:approx}
  Let $X$ be a locally finitary compact space, $\mathcal U$ be an open
  subset of $\Val_{\ast,\wk} X$, where $\ast$ is nothing or ``$\leq
  1$''.  For every $\nu \in \mathcal U$, there is a simple valuation
  $\nu'$ in $\mathcal U$ such that $\nu' \leq \nu$.
\end{lemma}
\begin{proof}
  When $\ast$ is nothing, this is
  Lemma~\ref{lemma:mu:simple:approx:1}, together with the fact that
  $\Val_\wk X = \Val_\pw X$.

  When $\ast$ is ``$\leq 1$'', we use the definition of the weak
  topology: $\nu$ is in some finite intersection
  $\bigcap_{i=1}^m [U_i > r_i]_{\leq 1}$ of subbasic open sets
  included in $\mathcal U$.  Then $\nu$ is also in the corresponding
  finite intersection $\bigcap_{i=1}^m [U_i > r_i]$ of subbasic open
  sets of $\Val_\wk X$.  We have just seen that there is a simple
  valuation $\nu' \leq \nu$ in $\bigcap_{i=1}^m [U_i > r_i]$.  Since
  $\nu' \leq \nu$, $\nu'$ is a subprobability valuation, so $\nu'$ is in
  $\bigcap_{i=1}^m [U_i > r_i]_{\leq 1}$, hence in $\mathcal U$.
\end{proof}

\section{\bf Capacities}
\label{sec:capacities}

Capacities are a generalization of valuations (or measures) introduced
by Choquet \cite{Choquet:capacities}, where modularity is abandoned in
favor of weaker properties.  We will need the following kind.

Given a subset $B$ of a topological space, the \emph{unanimity game}
$\ugame B \colon \Open X \to \creal$ maps every open set $U$ to $1$ if
$B \subseteq U$, to $0$ otherwise.  When $B = \{x\}$, $\ugame B$ is
simply the Dirac mass $\delta_x$, but in general $\ugame B$ is not
modular.

We will consider functions $\kappa$ of the form
$\sum_{x \in A} a_x \ugame {B_x}$, where $A$ is a finite subset of
$X$, and for each $x \in A$, $a_x$ is a number in $\Rp$ and $B_x$ is a
finite non-empty subset of $X$, which we call \emph{simple capacities}
here.  We compare capacities, and in general all functions from $\Open
X$ to $\creal$, by $\kappa \leq \nu$ if and only if $\kappa (U) \leq
\nu (U)$ for every $U \in \Open X$, extending the stochastic ordering
from continuous valuations to all maps.

In this setting, an element $f$ of $\Sigma \eqdef \prod_{x \in A} B_x$
is a function that maps each point $x \in A$ to an element $f (x)$ in
$B_x$. One can think of such functions $f$ as \emph{strategies} for
picking an element of $B_x$ for each $x \in A$.
We let $\Delta_\Sigma$ be the set of all families $\vec\beta \eqdef
{(\beta_f)}_{f \in \Sigma}$ of non-negative real numbers such that
$\sum_{f \in \Sigma} \beta_f = 1$.  $\Delta_\Sigma$ is simply the
standard $n$-simplex $\Delta_n \eqdef \{(\beta_0, \beta_1, \cdots,
\beta_n) \in \Rp^{n+1} \mid \sum_{i=0}^n \beta_i=1\}$, where $n$ is
the cardinality of $\Sigma$ minus $1$.

In order to show the following lemma, we will need to introduce the
Choquet integral $\int_{x \in X} h (x) d\nu$ of a lower semicontinuous
map $h \colon X \to \creal$ with respect to a set function
$\nu \colon \Open X \to \creal$.  By definition, this is equal to the
Riemann integral $\int_0^\infty \nu (h^{-1} (]t, \infty])) dt$.  Note
that this makes sense, because $h^{-1} (]t, \infty])$ is open for
every $t \in \Rp$, and because every non-increasing map is
Riemann-integrable.  In our setting, this form of the Choquet integral
was introduced by Tix \cite{Tix:bewertung}, and differs only slightly
from Choquet's original definition
\cite[Section~48]{Choquet:capacities}.  Tix proved that, when $\nu$ is
a continuous valuation, $\int_{y \in X} h (y) d\nu$ is linear and
Scott-continuous in $h$ \cite[Lemma~4.2]{Tix:bewertung}.  It is an
easy exercise to verify that
$\int_{y \in X} \chi_U (y) d\nu = \nu (U)$ for every open subset $U$
of $X$, where $\chi_U$ is the characteristic map of $U$.  It follows
that, when $h$ is of the form $\sum_{j=0}^m \alpha_j \chi_{U_j}$,
$\int_{y \in X} h (y) d\nu = \sum_{j=0}^m \alpha_j \nu (U_j)$.

For a simple capacity $\kappa \eqdef \sum_{x \in A} a_x \ugame {B_x}$,
we compute $\int_{y \in X} h (y) d\kappa$ as follows.  For each
$x \in A$,
$\int_{y \in X} h (y) d\ugame {B_x} = \int_0^\infty \ugame {B_x}
(h^{-1} (]t, \infty])) dt$ by the Choquet formula.  But
$\ugame {B_x} (h^{-1} (]t, \infty])) = 1$ if and only if
$B_x \subseteq h^{-1} (]t, \infty])$, if and only if
$\min_{y \in B_x} h (y) > t$.  It follows that
$\int_{y \in X} h (y) d\ugame {B_x} = \min_{y \in B_x} h (y)$.  Hence
$\int_{y \in X} h (y) d\kappa = \sum_{x \in A} a_x \min_{y \in B_x} h
(y)$.

\begin{lemma}
  \label{lemma:mincred}
  Let $X$ be a topological space, and $\kappa \eqdef \sum_{x \in A}
  a_x \ugame {B_x}$ be a simple capacity on $X$.

  Let $\nu$ be any bounded continuous valuation on $X$.  If
  $\kappa \leq \nu$, then, for some $\vec\beta \in \Delta_\Sigma$,
  $\sum_{f \in \Sigma, x \in A} \beta_f a_x \delta_{f (x)} \leq \nu$.
\end{lemma}
\begin{proof}
  This is a consequence of von Neumann's original minimax theorem
  \cite{vN:minimax}, which says that given any $n \times m$ matrix $M$
  with real entries,
  \begin{equation}
    \label{eq:minimax}
    \min_{\vec\alpha \in \Delta_m} \max_{\vec\beta \in \Delta_n}
    {\vec\beta}^\intercal M \vec\alpha
    = \max_{\vec\beta \in \Delta_n} \min_{\vec\alpha \in \Delta_m}
    {\vec\beta}^\intercal
    M \vec\alpha.
  \end{equation}
  In particular: $(\dagger)$ if for every $\vec\alpha \in \Delta_m$,
  there is a $\vec\beta \in \Delta_n$ such that
  ${\vec\beta}^\intercal M \vec\alpha \geq 0$ (namely, if the
  left-hand side of (\ref{eq:minimax}) is non-negative), then there is
  a $\vec\beta \in \Delta_n$ such that, for every
  $\vec\alpha \in \Delta_m$,
  ${\vec\beta}^\intercal M \vec\alpha \geq 0$.

  We first show that: $(*)$ given finitely many open subsets $U_0$, $U_1$,
  \ldots, $U_m$ of $X$, we can find a $\vec\beta \in \Delta_\Sigma$
  such that, for every $j$, $0\leq j \leq m$,
  $\sum_{f \in \Sigma, x \in A} \beta_f a_x \delta_{f (x)} (U_j) \leq
  \nu (U_j)$.

  Let $\kappa$ denote $\sum_{x \in A} a_x \ugame {B_x}$.  Since
  $\kappa \leq \nu$, for every lower semicontinuous map $h$,
  $\int_{y \in X} h (y) d\kappa = \int_0^\infty \kappa (h^{-1} (]t,
  \infty])) dt \leq \int_0^\infty \nu (h^{-1} (]t, \infty])) dt =
  \int_{y \in X} h (y) d\nu$.  In other words,
  $\sum_{x \in A} a_x \min_{y \in B_x} h (y) \leq \int_{y \in X} h (y)
  d\nu$.

  For every $\vec\alpha \in \Delta_m$, we consider
  $h_{\vec\alpha} \eqdef \sum_{j=0}^m \alpha_j \chi_{U_j}$ for $h$.
  The inequality we have just shown can be rewritten as
  $\sum_{x \in A} a_x \min_{y \in B_x} h_{\vec\alpha} (y) \leq
  \sum_{j=0}^m \alpha_j \nu (U_j)$.  For each $x \in A$, there is an
  element $y \in B_x$ that makes $h_{\vec\alpha} (y)$ minimal, and we
  call it $f_{\vec\alpha} (x)$.  Therefore
  $\sum_{x \in A} a_x h_{\vec\alpha} (f_{\vec\alpha} (x)) \leq
  \sum_{j=0}^m \alpha_j \nu (U_j)$.  By definition of
  $h_{\vec\alpha}$, and since
  $\chi_{U_j} (f_{\vec\alpha} (x)) = \delta_{f_{\vec\alpha} (x)}
  (U_j)$, this can be written equivalently as
  $\sum_{x \in A} \sum_{j=0}^m \alpha_j a_x \delta_{f_{\vec\alpha}
    (x)} (U_j) \leq \sum_{j=0}^m \alpha_j \nu (U_j)$.  It follows that
  there is a vector $\vec\beta$ in $\Delta_\Sigma$ such that, for
  every $j$, $0 \leq j\leq m$,
  $\sum_{j=0}^m \alpha_j \nu (U_j) - \sum_{j=0}^m \alpha_j \sum_{f \in
    \Sigma, x \in A} \beta_f a_x \delta_{f (x)} (U_j) \geq 0$: namely,
  $\beta_f \eqdef 1$ if $f=f_{\vec\alpha}$, and $\beta_f \eqdef 0$
  otherwise.

  That can also be written as
  $\sum_{f \in \Sigma, 0 \leq j \leq m} \alpha_j \beta_f \left(\nu
    (U_j) - \sum_{x \in A} a_x \delta_{f (x)} (U_j)\right) \geq 0$,
  hence as ${\vec\beta}^\intercal M \vec\alpha \geq 0$ for some matrix
  $M$.  Using $(\dagger)$, there is a vector
  $\vec\beta \in \Delta_\Sigma$ such that, for every
  $\vec\alpha \in \Delta_m$,
  ${\vec\beta}^\intercal M \vec\alpha \geq 0$, in other words
  $\sum_{j=0}^m \alpha_j \nu (U_j) - \sum_{j=0}^m \alpha_j \sum_{f \in
    \Sigma, x \in A} \beta_f a_x \delta_{f (x)} (U_j) \geq 0$.  In
  particular, for each $j$, $0\leq j\leq m$, taking $\vec\alpha$ such
  that $\alpha_j \eqdef 1$ and all its other components are $0$,
  $\nu (U_j) \geq \sum_{f \in \Sigma, x \in A} \beta_f a_x \delta_{f
    (x)} (U_j)$.  This proves $(*)$.

  For every finite family $\mathcal A$ of open subsets of $X$, let
  $C_{\mathcal A}$ be the set of vectors $\vec\beta \in \Delta_\Sigma$
  such that
  $\sum_{f \in \Sigma, x \in A} \beta_f a_x \delta_{f (x)} (U) \leq
  \nu (U)$ for every $U \in \mathcal A$.  Claim $(*)$ above states
  that $C_{\mathcal A}$ is non-empty (when $\mathcal A$ is non-empty;
  when $\mathcal A$ is empty, this is vacuously true).  It is also a
  closed subset of $\Delta_\Sigma$.  The family
  ${(C_{\mathcal A})}_{\mathcal A \in \Pfin (\Open X)}$ then has the
  finite intersection property: given any finite collection of
  elements $\mathcal A_1$, \ldots, $\mathcal A_k$ in
  $\Pfin (\Open X)$,
  $\bigcap_{i=1}^k C_{\mathcal A_i} = C_{\bigcup_{i=1}^k \mathcal
    A_i}$ is non-empty.  Since $\Delta_\sigma$ is compact, the
  intersection
  $\bigcap_{\mathcal A \in \Pfin (\Open X)} C_{\mathcal A}$ is
  non-empty.  Let $\vec\beta$ be any vector in that intersection.  For
  every $U \in \Open X$, since $\vec\beta$ is in $C_{\{U\}}$, we have
  $\sum_{f \in \Sigma, x \in A} \beta_f a_x \delta_{f (x)} (U) \leq
  \nu (U)$, and we conclude.
\end{proof}

\section{\bf The main theorem}
\label{sec:bf-main-theorem}

We come to our main theorem.  It applies in particular to every
quasi-continuous dcpo $X$, namely to every locally finitary compact
sober space, as we have announced; but sobriety is not needed.  We
spend the rest of the section proving it.
\begin{theorem}
  \label{thm:V:locfin}
  For every locally finitary compact space $X$,
  $\Val_\wk X=\Val_\pw X$ and
  $\Val_{\leq 1, \wk} X=\Val_{\leq 1, \pw} X$ are compact, locally
  finitary compact, sober spaces.  In particular, they are
  quasi-continuous dcpos and the weak topology coincides with the
  Scott topology.

  The sets of the form $\interior {\upc E}$, where $E$ ranges over the
  finite non-empty sets of simple (resp., simple subprobability)
  valuations form a base of the topology.
\end{theorem}

Let $\ast$ be nothing or ``$\leq 1$''.  We recall that the equality
$\Val_\wk X = \Val_\pw X$ holds for every locally finitary compact
space $X$, as shown by Heckmann \cite[Theorem~4.1]{heckmann96}.  The
equality $\Val_{\leq 1, \wk} X = \Val_{\leq 1, \pw} X$ immediately
follows from it.

We also recall that the quasi-continuous dcpos are exactly the locally
finitary compact sober spaces, and in particular that their topology
must be the Scott topology.  The fact that $\Val_{\ast,\wk} X$ is
compact follows from the fact that it has a least element in the
stochastic ordering, namely the zero valuation: every open cover
${(\mathcal U_i)}_{i \in I}$ of $\Val_{\ast,\wk}$ must be such that
some $\mathcal U_i$ contains the zero valuation, and therefore
coincide with the whole of $\Val_{\ast,\wk} X$, since open sets are
upwards-closed.

Finally, we recall that $\Val_{\ast, \wk} X$ is sober.

Therefore, it remains to show that $\Val_{\ast, \wk} X$ is locally
finitary compact.  In the rest of this section, we fix
$\nu \in \Val_{\ast,\wk} X$, and an open neighborhood $\mathcal U$ of
$\nu$ in the weak topology.  Then $\nu$ is in some finite intersection
$\bigcap_{i=1}^n [U_i > r_i]_\ast$ included in $\mathcal U$, where
each $U_i$ is open in $X$ and $r_i \in \Rp \diff \{0\}$.  We will find
a finite set $E$ of simple valuations and an open subset $\mathcal V$
of $\Val_{\ast, \wk} X$ such that
$\nu \in \mathcal V \subseteq \upc E \subseteq \mathcal U$.

Let us simplify the problem slightly.  By
Lemma~\ref{lemma:mu:simple:approx}, there is a simple valuation
$\nu' \leq \nu$ in $\bigcap_{i=1}^n [U_i > r_i]_\ast$. Hence, without loss
of generality, we may assume that $\nu$ itself is a simple valuation
$\sum_{x \in A} a_x \delta_x$, where $A$ is a finite subset of $X$,
and $a_x \in \Rp \diff \{0\}$ for every $x \in A$.

Since $\nu (U_i) > r_i$ for every $i$, $1\leq i\leq n$, there is a
number $a \in ]0, 1[$ such that $a . \nu (U_i) > r_i$ for every $i$.
There is also a positive number $s_i$ such that
$a . \nu (U_i) > s_i > r_i$.  We will need those numbers $a$ and $s_i$
only near the end of the proof.

Let us define a suitable open set $\mathcal V$.  For each point
$x \in A$, let $I_x \eqdef \{i \in I \mid x \in U_i\}$.  Then
$\bigcap_{i \in I_x} U_i \diff \dc (A \diff \upc x)$ is an open
neighborhood of $x$.  It is easy to see that $x$ is in
$\bigcap_{i \in I_x} U_i$, but perhaps a bit less easy to see that $x$
is not in $\dc (A \diff \upc x)$: otherwise there woud be an element
$y \in A \diff \upc x$ above $x$, and that is impossible.

Since $X$ is locally finitary compact, for each $x \in A$, one can
find a finite set $B_x$ such that
$x \in \interior {\upc B_x} \subseteq \upc B_x \subseteq \bigcap_{i
  \in I_x} U_i \diff \dc (A \diff \upc x)$.  We will require a bit
more, and we will make sure that $B_x$ is also included in
$\interior {\upc B_y}$ for every $y \in A$ such that $y \leq x$.  This
can be done by finding $B_x$ in stages, starting from the lowest
elements $x$ of $A$ and going up.  Formally, since $A$ is finite, we
define $B_x$ by course-of-values induction on the number of elements
$y \in A$ such that $y \leq x$, as follows: for each $x \in A$, we
simply find a finite set $B_x$ such that
$x \in \interior {\upc B_x} \subseteq \upc B_x \subseteq \bigcap_{i
  \in I_x} U_i \diff \dc (A \diff \upc x) \cap \bigcap_{y \in A, y<x}
\interior {\upc B_y}$, where the last term is available, and an open
neighborhood of $y \in A$ hence of $x$, by induction hypothesis.

We also define $V_x \eqdef \interior {\upc B_x}$.  We note the
following three facts.
\begin{lemma}
  \label{lemma:z}
  For all $x, y \in A$ with $x \leq y$, $\upc B_y \subseteq V_x
  \subseteq \upc B_x$.  \qed
\end{lemma}

\begin{lemma}
  \label{lemma:a}
  For every $x \in A$, for every $i \in I$, if $x \in U_i$ , then
  $B_x \subseteq U_i$.
\end{lemma}
\begin{proof}
  If $x \in U_i$, then $i \in I_x$.  Since
  $B_x \subseteq \bigcap_{i \in I_x} U_i$, the claim follows.
\end{proof}

\begin{lemma}
  \label{lemma:b}
  For all $x, y \in A$, $x \in V_y$ if and only if $y \leq x$.
\end{lemma}
\begin{proof}
  If $y \leq x$, and since $V_y$ is an open neighborhood of $y$, and
  is in particular upwards-closed, $x$ is also in $V_y$. If
  $y \not\leq x$, then $x$ is in $A \diff \upc y$, hence in
  $\dc (A \diff \upc y)$. It follows that $x$ cannot be in
  $\bigcap_{i \in I_y} U_i \diff \dc (A \diff \upc y)$, hence cannot
  be in the smaller set $V_y$.
\end{proof}

\begin{definition}[$\mathcal V$]
  \label{defn:V}
  Let $\pow_{\upc} A$ denote the (finite) family of upwards-closed
  subsets of $A$.  For each $B \in \pow_{\upc} A$, let
  $V_B \eqdef \bigcup_{x \in B} V_x$.  Let also
  $s_B \eqdef a . \sum_{x \in B} a_x$. The open set $\mathcal V$ is
  $\bigcap_{B \in \pow_{\upc} A} [s_B \ll V_B]$.
\end{definition}
Recall that $\mu \in [s_B \ll V_B]$ if and only if $s_B \ll \mu
(V_B)$, if and only if $s_B=0$ or $s_B < \mu (V_B)$.

\begin{lemma}
  \label{lemma:nu:V}
  $\nu \in \mathcal V$.
\end{lemma}
\begin{proof}
  For every $B \in \pow_{\upc} A$, we claim that $A \cap V_B = B$. For
  every $x \in B$, $V_x$ is included in $V_B$, and since $V_x$ is an
  open neighborhood of $x$, it follows that $x$ is in $V_B$; $x$ is
  also in $A$, since $B \subseteq A$.  Conversely, if
  $x \in A \cap V_B$, then $x$ is in $V_y$ for some $y \in B$.  Both
  $x$ and $y$ are in $A$, so by Lemma~\ref{lemma:b}, we obtain that
  $y \leq x$.  Since $B$ is upwards-closed, $x$ is in $B$.

  Let us verify that $\nu$ is in $\mathcal V$, namely that, for every
  $B \in \pow_{\upc} A$, $s_B \ll \nu (V_B)$.  Indeed,
  $\nu (V_B) = \sum_{x \in A \cap V_B} a_x = \sum_{x \in B} a_x$,
  since $A \cap V_B = B$.  Now, since $a < 1$,
  $a . \sum_{x \in B} a_x \ll \sum_{x \in B} a_x$.  In other words,
  $s_B \ll \nu (V_B)$, as desired.
\end{proof}

Finding the finite set $E$ is more difficult.  As a first step in that
direction, let $\kappa \eqdef a . \sum_{x \in A} a_x \ugame {B_x}$,
and let us consider the set $\mathcal Q$ of all the continuous
valuation $\mu \in \Val_\ast X$ such that $\kappa \leq \mu$.
\begin{lemma}
  \label{lemma:V:Q}
  $\mathcal V \subseteq \mathcal Q$.
\end{lemma}
\begin{proof}
  Let $\mu$ be any element of $\mathcal V$.  We must show that, for
  every open subset $U$ of $X$,
  $a . \sum_{x \in A, B_x \subseteq U} a_x \leq \mu (U)$.

  Let $B \eqdef \{x \in A \mid B_x \subseteq U\}$.  For every
  $x \in B$ and every $y \in A$ with $x \leq y$, we have
  $B_y \subseteq \upc B_x \subseteq U$ by Lemma~\ref{lemma:z}, so $y$
  is in $B$.  Hence $B$ is upwards-closed in $A$.

  Then the left-hand side $a . \sum_{x \in A, B_x \subseteq U}
  a_x$ is just $s_B$.  Since $\mu \in \mathcal V$, $s_B \ll \mu
  (V_B)$. We recall that $V_B = \bigcup_{x \in B} V_x$, that
  $V_x$ is included in $\upc B_x$ for each
  $x$, and that (by the definition of $B$), $\upc
  B_x$ is included in $U$ for every $x \in B$. Therefore $V_B
  \subseteq U$, and hence $\mu (V_B) \leq \mu
  (U)$, which concludes the proof.
\end{proof}

Let $\Sigma \eqdef \prod_{x \in A} B_x$, and $\Delta_\Sigma$ be the
associated standard simplex.  Lemma~\ref{lemma:V:Q}, together with
Lemma~\ref{lemma:mincred}, immediately implies the following.
\begin{lemma}
  \label{lemma:V:beta}
  Every element $\mu$ of $\mathcal V$ is above a simple valuation of
  the form
  $a. \sum_{f \in \Sigma, x \in A} \beta_f a_x \delta_{f (x)}$, for
  some $\vec\beta \in \Delta_\Sigma$.
\end{lemma}

Let $E_0$ be the set of simple valuations obtained this way, namely
the set of simple valuations
$a . \sum_{f \in \Sigma, x \in A} \beta_f a_x \delta_{f (x)}$, where
$\vec\beta \in \Delta_\Sigma$.  We have just shown that every element
$\mu$ of $\mathcal V$ is above some element of $E_0$.

Note that the elements $\varpi$ of $E_0$ can all be written as
$\sum_{z \in Z} c_z \delta_z$, where $Z \eqdef \bigcup_{x \in A} B_x$,
and $c_z \in \Rp$.  For each such $\varpi$, let $\overline\varpi$ be
$\sum_{z \in Z} \frac 1 N \lfloor N c_z \rfloor$, where $N$ is a
fixed, large enough (in particular, non-zero) natural number that we
will determine shortly.

\begin{definition}[$E$]
  \label{defn:E}
  $E$ is the set of all simple valuations $\overline\varpi$, where
  $\varpi$ ranges over $E_0$.
\end{definition}

\begin{lemma}
  \label{lemma:E:fin}
  $E$ is a finite set.
\end{lemma}
\begin{proof}
  $Z$ is finite and the coefficients $\frac 1 N \lfloor N c_z \rfloor$
  are integer multiples of $\frac 1 N$ between $0$ and
  $\sum_{x \in A} a_x$.
\end{proof}

\begin{lemma}
  \label{lemma:overline:varpi}
  $\mathcal V \subseteq \upc E$.
\end{lemma}
\begin{proof}
  For every $z \in Z$, and every $c_z \in \Rp$,
  $\frac 1 N \lfloor N c_z \rfloor \leq c_z$.  It follows that
  $\overline\varpi \leq \varpi$ for every $\varpi \in E_0$.  Since
  every element of $\mathcal V$ is above some element $\varpi$ of
  $E_0$ by Lemma~\ref{lemma:V:beta}, it is also above the
  corresponding element $\overline\varpi$ of $E$.
\end{proof}

\begin{lemma}
  \label{lemma:upE:U}
  $\upc E \subseteq \mathcal U$.
\end{lemma}
\begin{proof}
  We show that $E$ is included in $\bigcap_{i=1}^n [U_i >
  r_i]_\ast$. For every $x \in A$, for every $y \in B_x$, we have
  $\delta_y \geq \ugame {B_x}$, simply because every open neighborhood
  of $B_x$ must contain $x$. Hence, for every $\varpi \in E_0$, say
  $\varpi = a . \sum_{f \in \Sigma, x \in A} \beta_f a_x \delta_{f
    (x)}$, where $\vec\beta \in \Delta_\Sigma$, we have
  $\varpi \geq a . \sum_{f \in \Sigma, x \in A} \beta_f a_x \ugame
  {B_x} = a. \sum_{x \in A} (\sum_{f \in \Sigma} \beta_f) a_x \ugame
  {B_x} = a . \sum_{x \in A} a_x \ugame {B_x} = \kappa$.  For every
  $i$, $1\leq i\leq n$, Lemma~\ref{lemma:a} states that for every
  $x \in A$, if $x \in U_i$ then $B_x$ is included in $U_i$.
  Therefore
  $\varpi (U_i) = a. \sum_{x \in A, B_x \subseteq U_i} a_x \geq
  a. \sum_{x \in A \cap U_i} a_x = a. \nu (U_i)$.  We now remember
  that $a . \nu (U_i) > s_i > r_i$.  In particular,
  $\varpi (U_i) > s_i$.

  It is time we fixed the value of $N$.  The values of $c_z$ and of
  $\frac 1 N \lfloor N c_z \rfloor$ differ by $\frac 1 N$ at most, so
  for any open set $U$, the values $\varpi (U)$ and
  $\overline\varpi (U)$ differ by $\frac 1 N |Z|$ at most, where $|Z|$
  is the cardinality of $Z$.  It follows that
  $\overline\varpi (U_i) > s_i - \frac 1 N |Z|$.  By picking any
  non-zero natural number $N$ larger than or equal to
  $\frac {|Z|} {s_i-r_i}$ for every $i$, $1\leq i\leq n$, we therefore
  ensure that $\overline\varpi (U_i) > r_i$ for every $i$, hence that
  $\overline\varpi$ is in $\mathcal U$.  Since that holds for every
  $\varpi \in E_0$, $E$ is included in $\mathcal U$, hence also
  $\upc E$.
\end{proof}

Hence, as promised, $\nu \in \mathcal V$ (Lemma~\ref{lemma:nu:V})
$\subseteq \upc E$ (Lemma~\ref{lemma:overline:varpi})
$\subseteq \mathcal U$ (Lemma~\ref{lemma:upE:U}), where $\mathcal V$
is open (Definition~\ref{defn:V}) and $E$ is finite
(Lemma~\ref{lemma:E:fin}).  This concludes the proof of
Theorem~\ref{thm:V:locfin}.  \qed

\section{The case of probability continuous valuations}
\label{sec:case-prob-cont}

We now apply the previous results to the space $\Val_{1,\wk} X$ of
probability continuous valuations.  A space $X$ is \emph{pointed} if
and only if it has a least element $\bot$ in its specialization
preordering.  We are not assuming $X$ to be $T_0$, so $\dc \bot$ is a
closed subset that may be different from $\{\bot\}$.  The open subsets
of $X \diff \dc \bot$ are just the proper open subsets of $X$.

The following is \emph{Edalat's lifting trick}, which was introduced
in \cite[Section~3]{edalat95a} for dcpos, and in
\cite[Section~7.4]{alvarez00} for stably locally compact spaces.
Every continuous valuation $\nu$ on $X$ gives rise to a continuous
valuation $\nu^-$ on $X \diff \dc \bot$ by $\nu^- (U) \eqdef \nu (U)$
for every $U \in \Open (X \diff \dc \bot)$.  If $\nu \in \Val_1 X$,
then $\nu^-$ is in $\Val_{\leq 1} X$, and we have much more, as we now
show.
\begin{lemma}
  \label{lemma:VwX:pointed}
  Let $X$ be a pointed topological space, with least element $\bot$.
  The map $\nu \mapsto \nu^-$ is a homeomorphism of $\Val_{1,\wk} X$
  onto $\Val_{\leq 1,\wk} (X \diff \dc\bot)$.  Its inverse maps every
  subprobability continuous valuation $\mu$ on $X \diff \dc \bot$ to
  $\mu^+$, defined by
  $\mu^+ (U) \eqdef \mu (U \diff \dc \bot) + (1-\mu (X \diff \dc
  \bot)) \delta_\bot$, for every $U \in \Open (X \diff \dc \bot)$.
\end{lemma}
\begin{proof}
  Let $\nu \in \Val_1 X$.  For every $U \in \Open X$,
  ${(\nu^-)}^+ (U) = \nu^- (U \diff \dc \bot) + (1-\nu^- (X \diff \dc
  \bot)) \delta_\bot (U)$.  If $U$ is a proper open subset of $X$,
  then $U$ does not contain $\bot$, so $U \diff \dc \bot = U$, and
  $\delta_\bot (U)=0$, so ${(\nu^-)}^+ (U) = \nu^- (U) = \nu (U)$.  If
  $U=X$, then
  ${(\nu^-)}^+ (U) = \nu^- (X \diff \dc \bot) + (1-\nu^- (X \diff \dc
  \bot)) = 1$, and this is equal to $\nu (U)$ since $U=X$ and
  $\nu \in \Val_1 X$.

  For every $U \in \Open (X \diff \dc \bot)$,
  ${(\mu^+)}^- (U) = \mu^+ (U) = \mu (U)$, since
  $U \diff \dc \bot = U$, and $\bot$ is not in $U$.

  Hence the two maps $\nu \mapsto \nu^-$ and $\mu \mapsto \mu^+$ are
  inverse of each other.
  
  For every open subset $U$ of $X$ and every $r \in \Rp \diff \{0\}$,
  the inverse image of $[U > r]_1$ by $\mu \mapsto \mu^+$ is equal to
  one of the following sets.  If $U=X$ and $r<1$, this is the whole of
  $\Val_{\leq 1,\wk} X$.  If $U=X$ and $r\geq 1$, this is empty.
  Finally, if $U$ is a proper subset of $X$, hence does not contain
  $\bot$, then this is the set of all
  $\mu \in \Val_{\leq 1} (X \diff \dc \bot)$ such that $\mu^+ (U) >
  r$, where $\mu^+ (U) = \mu (U)$: hence this is $[U > r]_{\leq 1}$.
  In any case, that inverse image is open, so $\mu \mapsto \mu^+$ is
  continuous.
  
  For every open subset $U$ of $X \diff \dc\bot$, for every
  $r \in \Rp \diff \{0\}$, the inverse image of $[U > r]_{\leq 1}$ by
  $\nu \mapsto \nu^-$ is $[U > r]_1$.  Therefore $\nu \mapsto \nu^-$
  is continuous.
\end{proof}

Lemma~\ref{lemma:VwX:pointed} allows us to obtain the following
corollary to Theorem~\ref{thm:V:locfin}.
\begin{corollary}
  \label{corl:V1:fin}
  For every locally finitary compact, pointed space $X$,
  $\Val_{1, \wk} X$ is compact, locally finitary compact, and sober.
  In particular, it is a quasi-continuous dcpo, and the weak topology
  coincides with the Scott topology.

  The sets of the form $\interior {\upc E}$, where $E$ ranges over the
  finite non-empty sets of simple probability
  valuations form a base of the topology.  \qed
\end{corollary}

\section{The Scott and weak topologies may differ}
\label{sec:scott-weak-topol}

The Scott and weak topologies on $\Val_\ast X$ seem to agree in many
situations, and Alvarez-Manilla, Jung and Keimel asked whether they
agree on $\Val_{\leq 1} X$ for every stably compact space $X$
\cite[Section~5, second open problem]{AMJK:scs:prob}.  We show that
this is not the case.

Let $\alpha (\nat)$ be the one-point compactification of the discrete
space $\nat$.  Its elements are the natural numbers, plus a fresh
element $\infty$.  Its open subsets are the subsets of $\nat$ (not
containing $\infty$), plus all the subsets $\alpha (\nat) \diff E$,
where $E$ ranges over the finite subsets of $\nat$.  A \emph{discrete
  valuation} on $\alpha (\nat)$ is any valuation of the form
$\sum_{n \in \nat} a_n \delta_n + a_\infty \delta_\infty$, where each
$a_n$ and $a_\infty$ are in $\creal$.  They are all continuous
valuations.
\begin{lemma}
  \label{lemma:aN:discrete}
  Letting $\ast$ be ``$\leq 1$'' or ``$1$''.
  \begin{enumerate}[label=(\roman*)]
  \item The continuous valuations $\nu$ on $\alpha (\nat)$ are
    exactly the discrete valuations.
  \item The function
    $f \colon \Val_\ast (\alpha (\nat)) \to Y_\ast$ that maps
    $\sum_{n \in \nat} a_n \delta_n + a_\infty \delta_\infty$ to
    ${(a_x)}_{x \in \alpha (\nat)}$ is an order-isomorphism onto the
    poset $Y_\ast$ of families of non-negative real numbers whose sum
    is at most $1$ (if $\ast$ is ``$\leq 1$'') or exactly $1$ (if
    $\ast$ is ``$1$''), ordered pointwise.
  \item The set $\mathcal V$ of families
    ${(a_x)}_{x \in \alpha (\nat)}$ of $Y_\ast$ such that
    $a_\infty > 0$ is Scott-open in $Y_\ast$, but
    $f^{-1} (\mathcal V)$ does not contain any basic open neighborhood
    $\bigcap_{i=1}^n [U_i > r_i]_\ast$ of $\delta_\infty$.
  \end{enumerate}
\end{lemma}
\begin{proof}
  (i) Let $\nu$ be any continuous valuation over $\alpha (\nat)$.  We
  recall that every continuous valuation on an LCS-complete space
  extends to a measure on the Borel $\sigma$-algebra
  \cite[Theorem~1.1]{dBGLJL:LCS}.  Every locally compact sober space
  is $G_\delta$ in itself, hence LCS-complete.  Since every Hausdorff
  space is sober, and clearly locally compact, $\alpha (\nat)$ is
  LCS-complete, and therefore $\nu$ extends to a measure $\tilde\nu$
  on the Borel $\sigma$-algebra of $\alpha (\nat)$.  It is easy to see
  that the latter $\sigma$-algebra is the whole of $\pow (\nat)$.  We
  define $a_n \eqdef \tilde\nu (\{n\}) = \nu (\{n\})$, and
  $a_\infty \eqdef \tilde\nu (\{\infty\})$.  By $\sigma$-additivity,
  for every (necessarily measurable) subset $E$ of $\alpha (\nat)$,
  $\tilde\nu (E) = \sum_{x \in E} a_x$.  In particular, for every open
  subset $U$ of $\alpha (\nat)$,
  $\nu (U) = \sum_{x \in U} a_x = (\sum_{n \in \nat} a_n \delta_n +
  a_\infty \delta_\infty) (U)$.

  (ii) Let $\nu$ be any element of $\Val_\ast (\alpha (\nat))$, and
  $\tilde\nu$ be a measure that extends $\nu$ to the Borel
  $\sigma$-algebra.  In a more precise way as in the statement of the
  lemma, we define $f (\nu)$ as ${(a_x)}_{x \in \alpha (\nat)}$, as
  given in item~(i), so that
  $\nu = \sum_{n \in \nat} a_n \delta_n + a_\infty \delta_\infty$.
  This defines a bijection $f$ of $\Val_\ast (\alpha (\nat))$ onto
  $Y_\ast$.

  Since $\{\infty\} = \bigcap_{n \in \nat} V_n$, where $V_n$ is the
  open set $\{n, n+1, \cdots, \infty\}$, and since $\tilde\nu$ is a
  bounded measure,
  $a_\infty = \tilde\nu (\{\infty\}) = \inf_{n \in \nat} \tilde\nu
  (V_n) = \inf_{n \in \nat} \nu (V_n)$.  This implies that $a_\infty$
  grows as $\nu$ grows.  It is clear that $a_n = \nu (\{n\})$ grows,
  too, as $\nu$ grows.  Therefore $f$ is monotonic, and its inverse is
  clearly monotonic as well.  (This discussion is superfluous when $\ast$ is
  ``$1$'', by the way, since in that case the ordering on
  $\Val_1 (\alpha (\nat))$ and on $Y_1$ is just equality.)

  (iii) $\mathcal V$ is clearly Scott-open in $Y_\ast$.  We now
  imagine that $f^{-1} (\mathcal V)$ contains a basic open
  neighborhood $\bigcap_{i=1}^n [U_i > r_i]_\ast$ of $\delta_\infty$,
  where each $U_i$ is open in $\alpha (\nat)$, and
  $r_i \in \Rp \diff \{0\}$.  Since
  $\delta_\infty \in [U_i > r_i]_\ast$, $U_i$ must contain $\infty$
  (and $r_i < 1$), so $U_i = \alpha (\nat) \diff E_i$ for some finite
  subset $E_i$ of $\nat$.  Let $n$ be a natural number that is not in
  any of the finite sets $E_i$, $1\leq i\leq n$.  Then
  $\delta_n (U_i) = 1 > r_i$, so $\delta_n$ is in
  $\bigcap_{i=1}^n [U_i > r_i]_\ast$, hence in $f^{-1} (\mathcal V)$.
  However, $f (\delta_n)$ is the family
  ${(a_x)}_{x \in \alpha (\nat)}$ such that $a_x=0$ for every
  $x \in \alpha (\nat)$ except for $a_n=1$; in particular,
  $a_\infty=0$, showing that $f (\delta_n)$ is not in $\mathcal V$:
  contradiction.
\end{proof}

\begin{theorem}
  Let $\ast$ be nothing, ``$\leq 1$'' or ``$1$''.  The Scott topology
  on $\Val_\ast (\alpha (\nat))$ is strictly finer than the weak topology.
\end{theorem}
\begin{proof}
  We recall that the Scott topology on any space of the form
  $\Val_\ast X$ is always finer than the weak topology.
  
  When $\ast$ is ``$\leq 1$'' or ``$1$'', this is
  Lemma~\ref{lemma:aN:discrete}, item~(iii): $f^{-1} (\mathcal V)$ is
  a Scott-open neighborhood of $\delta_\infty$ in
  $\Val_\ast (\alpha (\nat))$ that is not open in the weak topology.

  When $\ast$ is nothing, we notice that
  $\Val_{\leq 1} (\alpha (\nat))$ is Scott-closed in
  $\Val (\alpha (\nat))$.  This easily implies that the Scott topology
  on $\Val_{\leq 1} (\alpha (\nat))$ is the subspace topology induced
  by the Scott topology on $\Val (\alpha (\nat))$.  If the latter
  agreed with the weak topology, then the Scott topology on
  $\Val_{\leq 1} (\alpha (\nat))$ would be the subspace topology
  induced by the inclusion in $\Val_\wk (\alpha (\nat))$.  But the
  latter is just the weak topology on $\Val_{\leq 1} (\alpha (\nat))$,
  and we have just seen that it differs from the Scott topology.
\end{proof}
The gap between the Scott and weak topologies on
$\Val_1 (\alpha (\nat))$ is really enormous.  By Corollary~37 of
\cite{AMJK:scs:prob}, $\Val_{\leq 1} X$ and $\Val_1 X$ are stably
compact for any stably compact space $X$.  This applies to
$X \eqdef \alpha (\nat)$, since every compact Hausdorff space is
stably compact.  One checks easily (e.g., by using
Lemma~\ref{lemma:aN:discrete}, item~(ii)) that the stochastic ordering
on $\Val_1 (\alpha (\nat))$ is simply equality, hence that the Scott
topology is the discrete topology.  But the only discrete spaces that
are (stably) compact are finite, and $\Val_1 (\alpha (\nat))$ is far
from finite.

The coincidence of the Scott and weak topologies of
Theorem~\ref{thm:V:locfin}, and first observed by Kirch in the case
where $X$ is a continuous dcpo, is probably exceptional.  We leave
open the question of the exact characterization of those spaces $X$
for which the weak and Scott topologies agree on $\Val_\ast X$.

\section*{Acknowledgments}
\label{sec:acknowledgments}

My deepest thanks to Xiaodong Jia, who found a mistake in a previous
version of Lemma~\ref{lemma:mincred}, and another one in a previous
version of Lemma~\ref{lemma:V:Q}.

% Section~\ref{sec:bf-main-theorem} greatly benefited from
% suggestions due to the referee, whom we thank here.

\bibliographystyle{plain}
\ifarxiv

\else
\bibliography{qcontval}
\fi

\end{document}